\documentclass[11pt, leqno, twoside]{article}

\usepackage{amssymb}
\usepackage{amsmath}
\usepackage{amsthm}
\usepackage{amsfonts}
\usepackage{mathrsfs}
\usepackage{indentfirst}

\usepackage{color}

\allowdisplaybreaks

\usepackage{txfonts}
\def\rr{{\mathbb R}}
\def\rn{{{\rr}^n}}

\def\zz{{\mathbb Z}}
\def\cc{{\mathbb C}}
\def\nn{{\mathbb N}}
\def\cp{{\mathcal P}}

\def\cm{{\mathcal M}}

\def\cd{{\mathcal D}}
\def\cs{{\mathcal S}}

\def\fz{\infty}
\def\az{\alpha}

\def\dz{\delta}

\def\lz{\lambda}

\def\lf{\left}
\def\r{\right}

\def\hs{\hspace{0.26cm}}

\def\ls{\lesssim}
\def\gs{\gtrsim}

\def\noz{\nonumber}
\def\wz{\widetilde}

\def\cs{{\mathcal S}}

\def\supp{\mathop\mathrm{\,supp\,}}

\def\loc{{\mathop\mathrm{\,loc\,}}}
\def\esinf{\mathop\mathrm{\,ess\,inf\,}}
\def\esup{\mathop\mathrm{\,ess\,sup\,}}

\def\q1{\wz q}
\def\Q1{q_1}
\def\vlp{{L^{p(\cdot)}(\rn)}}

\def\loc{{\mathop\mathrm{loc\,}}}

\newtheorem{thm}{Theorem}[section]
\newtheorem{prop}[thm]{Proposition}
\newtheorem{lem}[thm]{Lemma}

\theoremstyle{definition}
\newtheorem{defn}[thm]{Definition}
\newtheorem{rem}[thm]{Remark}

\numberwithin{equation}{section}
\numberwithin{equation}{section}

\begin{document}

\title{\bf\Large The Molecular Characterizations of Variable Triebel-Lizorkin Spaces
Associated with the Hermite Operator and Its Applications
\footnotetext{\hspace{-0.35cm} 2020 {\it
Mathematics Subject Classification}. 42B35, 47B15
\endgraf {\it Key words and phrases}. Hermite Operator, variable exponent,
Triebel-Lizorkin space, molecular decomposition.
\endgraf This project is supported by
the Natural Science Foundation of Changsha (Grant No. kq2202237)
and of Hunan province (Grant No. 2022JJ30371), the National
Natural Science Foundation of China
(Grant Nos. 11831007, 11871100) and Scientific Research Fund of Hunan Provincial Education
Department (Grant No. 21A0067).}}
\author{Qi Sun and Ciqiang Zhuo\,\footnote{Corresponding author}}
\date{ }
\maketitle

\vspace{-0.7cm}

\begin{center}
\begin{minipage}{13cm}
{\small {\bf Abstract}\quad
In this article, we introduce inhomogeneous
variable Triebel-Lizorkin spaces, $F_{p(\cdot),q(\cdot)}^{\alpha(\cdot),H}(\mathbb R^n)$,
associated with the Hermite operator $H:=-\Delta+|x|^2$,
where $\Delta$ is the Laplace operator on $\mathbb R^n$, and mainly establish the molecular
characterization of this space.
As applications, we obtain some regularity results to fractional Hermite equations
$$(-\Delta+|x|^2)^\sigma u=f,\quad (-\Delta+|x|^2+I)^\sigma u=f,$$
and the boundedness of spectral multiplier associated to the
operator $H$ on the variable Triebel-Lizorkin space
$F_{p(\cdot),q(\cdot)}^{\alpha(\cdot),H}(\mathbb R^n)$.
Furthermore, we explain the relationship between
$F_{p(\cdot),q(\cdot)}^{\alpha(\cdot),H}(\mathbb R^n)$ and the
variable Triebel-Lizorkin spaces $F_{p(\cdot),q(\cdot)}^{\alpha(\cdot)}(\mathbb R^n)$
(introduced in Diening t al. J. Funct. Anal. 256(2009), 1731-1768.) via
the atomic decomposition.
}
\end{minipage}
\end{center}

\vspace{0.1cm}
\section{Introduction}
In last decades, spaces of functions or distributions have played a prominent role in various
areas of mathematics such as harmonic analysis, partial differential equations, approximation
theory,
probability and so on.

Particularly, function spaces with variable exponents on $\rn$ attract more and more attention
recently. Indeed, they have been the subject of intensive since the early
work \cite{C-U03} of Cruz-Uribe and \cite{Di04} of Diening, because of their intrinsic
interest for applications in harmonic analysis \cite{cfbook,cw14,dhr11,jzzw19,NS12,yyyz16},
in partial differential equations and variation calculus \cite{am05,hhl08,su09},
in fluid dynamics \cite{am02} and image processing \cite{cgzl08}.
Recall that the variable Lebesgue space $\vlp $
is a generalization of the classical Lebesgue space $L^p(\rn)$, via replacing the
constant exponent $p$ by the exponent function $p(\cdot):\ \rn\to(0,\fz)$, which consists of
all measurable functions $f$ such that
$\int_\rn[|f(x)|/\lz]^{p(x)}\,dx<\fz$ for some $\lz\in(0,\fz)$.
As a generalization of both the variable Lebesgue space $\vlp $
and the classical Hardy space $H^{p}(\rn)$,
the variable Hardy spaces $H^{p(\cdot)}(\rn)$ were first introduced by
Nakai and Sawano \cite{NS12} with $p(\cdot)$ satisfying the globally log-H\"older continuous
condition and investigated in \cite{Sa13, zyl16} later.
Independently, Cruz-Uribe and Wang \cite{cw14} also studied
the variable Hardy spaces $H^{p(\cdot)}(\rn)$ with $p(\cdot)$ satisfying
some conditions slightly weaker than those used in \cite{NS12}. In addition,
the characterizations of $H^{p(\cdot)}(\rn)$ via Riesz transforms with $p(\cdot)$
satisfying the same conditions as in \cite{cw14} were presented in \cite{yzn16}.

Based on the classical Triebel-Lizorkin space,
Diening et al. \cite{DHR09} introduced the variable
Triebel-Lizorkin space $F_{p(\cdot),q(\cdot)}^{s(\cdot)}(\rn)$ via the sequence space
$L^{p(\cdot)}(\ell^{q(\cdot)}(\rn))$,
which was the first work mixing up the concept of function spaces with variable smoothness
and variable integrability. After that, Almeida and H\"ast\"o \cite{ah10} introduced the
variable Besov space via subtly introducing
the sequence space $\ell^{q(\cdot)}(L^{p(\cdot)}(\rn))$, which makes a further step in
completing the unification process of function spaces with variable smoothness and
integrability.
Later, those spaces were further investigated in \cite{gkv18,gm18,kv12} and
generalized into variable Besov-type and Triebel-Lizorkin-type spaces
(see, for instance, \cite{Dd12,wyyz18,yzy15,yzy151}).
Here we point out that the vector-valued convolutional inequalities on
the mixed Lebesgue sequence spaces $L^{p(\cdot)}(\ell^{q(\cdot)}(\rn))$ and
$\ell^{q(\cdot)}(L^{p(\cdot)}(\rn))$ developed, respectively, in
\cite[Theorem 3.2]{DHR09} and \cite[Lemma 4.7]{ah10} supply well remedy for the absence of
the Fefferman-Stein vector-valued inequality
on these spaces, in studying Triebel-Lizorkin(-type) spaces and Besov(-type) spaces with
variable smoothness and integrability.

Note that the classical Besov and Triebel-Lizorkin space can be viewed as spaces
associated to Laplace operator $\Delta$ or their square roots on $\rn$
(see \cite{bpt95,bpt96,bpt97}) and are widely used in harmonic
analysis and partial differential equations.
However, with the development of harmonic analysis and partial differential equations,
people realized that the classical Besov and Triebel-Lizorkin spaces are
not always the most suitable spaces to study a number of operators.
Motivated by this, the theory of Besov and Triebel-Lizorkin spaces
associated to operators has been studied intensively
by many authors \cite{BDY12,kv12,PX08} and the
references therein.
Particularly, in this paper, we would like to focus our interest on the Hermite operator
$H=-\Delta+|x|^2$ on the Euclidean space $\rn$,
which has attracted much attention as a specific case of the Schr\"odinger
operator $\mathcal{L}:=-\Delta+V$ by taking the nonnegative $V(x)=|x|^2$.
Specifically, Petrushev and Xu \cite{PX08} extended the fundamental and remarkable
work of Frazier and Jawerth in \cite{FJ90} to the
inhomogeneous Besov and Triebel-Lizorkin spaces in Hermite settings.
In \cite{BDY12}, the theory of homogeneous Besov spaces associated to operators
enjoying the Gaussian upper bounds was studied for the limited indices $\alpha \in(-1,1)$
and $p,\, q\in [1,\fz]$. By using the functional calculus techniques,
Keryacharian and Petrushev in \cite{kp15} constructed a space of test functions
and the corresponding distribution space associated to nonnegative
self-adjoint operators $L$ which satisfy Gaussian upper bound with the small time,
the Markov property and the H\"older continuity, and then introduced
the inhomogeneous Besov and Triebel-Lizorkin spaces associated to the operator $L$.
Moreover, some basic properties such as the embedding theorem, frame decompositions and heat
kernel characterizations of those spaces are obtained in \cite{kp15}.
This was further generalized into the Besov-type and Triebel-Lizorkin-type space
on a space of homogeneous type associated to the operator $L$ by Liu et al. in \cite{lyy16}.
Combining with variable function space, Zhuo and Yang \cite{ZY20}
investigated the variable Besov space $B_{p(\cdot),q(\cdot)}^{s(\cdot),L}(\mathcal{X})$,
associated with the operator $L$ as in \cite{lyy16} on a homogeneous space $\mathcal{X}$,
which is also a continuous work on variable function space associated to operators in
\cite{yz18, yzz17,yz16, zy15,zy19}.

Using a different approach from those used in \cite{kp15,PX08}, Bui and Duong \cite{BD15}
introduced both homogeneous and inhomogeneous Besov and Triebel-Lizorkin spaces associated
to Hermite operator,
in terms of heat kernels via square functions,
and established their molecular decomposition as well as some embeddings.
As applications,
the boundedness of singular operators, such as negative powers, spectral multipliers and
lift operators involving Hermite operators, on these function spaces were considered.
Moreover, in \cite{bd21}, the boundedness of higher-order Riesz transforms on Besov and
Triebel-Lizorkin spaces associated the Hermite operator was established and further applied to
prove certain regularity estimates of second-order elliptic equations in divergence
form with the
oscillator perturbations.

Motivated by \cite{BD15,bd21}, in this article,
we consider the variable Triebel-Lizorkin space associated with the Hermite operator $H$.
We mainly establish the molecular decomposition of this space. As applications,
we show the boundedness of negative powers and spectral multipliers of the Hermite operators
on some appropriate Triebel-Lizorkin spaces.

This article is concluded as follows. In Section \ref{se-2},
we recall some preliminary results and give the definition of the space
$F_{p(\cdot),q(\cdot)}^{\az(\cdot),H}(\rn)$.
In Section \ref{se-3}, the molecular characterization of the space
$F_{p(\cdot),q(\cdot)}^{\az(\cdot),H}(\rn)$ is established in
Theorems \ref{thm10-18a} and \ref{thm10-9} below.
As applications, in Section \ref{se-4}, we develop the boundedness of
the Riesz potential and the Bessel potential of the Hermite operator,
and the spectral multipliers of Laplace type
for the Hermite operator.
In the last section, we discuss the relationship between variable Triebel-Lizorkin
space $F_{p(\cdot),q(\cdot)}^{\az(\cdot)}(\rn)$ introduced in \cite{DHR09} and the
variable Triebel-Lizorkin space associated with the Hermite operator $H$ studied in the present
article, based on the atomic decomposition of the space
$F_{p(\cdot),q(\cdot)}^{\az(\cdot)}(\rn)$.

As usual, we write $f\ls g$ or $g\gs f$ if there exist $C>0$ with $f\le Cg$ and
always use $C$ or $c$ to denote positive constants
that are independent of the main parameters involved but whose values may differ from
line to line. If $f\ls g$ and $g\ls f$, we then write $f\sim g$.
For $a\in\rr$, we denote the integer part of $a$ by $\lfloor a\rfloor$.
Let $\cs(\rn)$ denote the Schwartz space of all complex-valued,
rapidly decreasing and infinitely differentiable functions on $\rn$ and by $\cs'(\rn)$
we denote the dual space of $\cs(\rn)$ equipped with the weak-$\ast$ topology.
In addition, we also use the following notations in our paper, we denote by $\rr$ the set of all
real numbers and $\rr_+$ the set of positive real numbers, $\rn$ denotes the $n$-dimensional
Euclidean space. Let $\nn:=\{1,2,\cdots\}$ and $\nn_0=\nn \cup\{0\}$.
If $E$ is a subset of $\rn$, we denote by $\textbf{1}_{E}$ is
\emph{characteristic function}. If $\alpha=(\alpha_1,\dots,\alpha_n)\in \nn_0^n$ and
$f\in \cs(\rn)$, let
$$D^\alpha f(x):=\frac{\partial^{|\alpha|}f(x)}{\partial x_1^{\alpha_1}
\cdots\partial x_n^{\alpha_n}},\quad \forall\,x\in\rn,$$
where $|\alpha|:=\alpha_1+\cdot+\alpha_n$.

\section{Preliminary\label{se-2}}
\subsection{Hermite Operators and Kernels }

We first recall some background on the Hermite
operator in \cite{ST03,Th93}. The Hermite operator
on $\rn$ is defined by
\begin{equation}\label{3-22a}
H=-\Delta+|x|^2,
\end{equation}
which is also called $n$-dimensional harmonic oscillator.
The eigenfunctions of the harmonic oscillator are the multi-dimensional Hermite functions
$$\textbf{h}_\alpha(x):=h_{\alpha_1}(x_1)h_{\alpha_2}(x_2)\cdot\,\cdots\,\cdot h_{\alpha_n}(x_n),
\quad \forall\, x=(x_1,\cdots,x_n)\in \rn,$$
where $\alpha=(\alpha_1,\cdots,\alpha_n)\in\nn_0^n$ and, for any
$m\in\nn$, the Hermite function $h_m$ of degree $m$ is defined by
$$h_m(t):=(2^mm!\sqrt{\pi})^{-1/2}H_m(t)e^{-t^2/2},\quad \forall\,t\in\rr,$$
here and thereafter, $H_m$ is the Hermite polynomial defined by
$$H_m(t):=(-1)^me^{t^2}\partial_t^m(e^{-t^2}),\ \forall t\in \rr.$$
It is known that $\{\textbf{h}_\alpha\}_{\az\in\nn_0^n}$ forms
an complete orthogonal basis for $L^2(\rn)$. Moreover, one has
$$L(\textbf{h}_\alpha)=(2|\alpha|+n)\textbf{h}_\alpha,$$
where $|\alpha|:=\az_1+\cdots+\az_n$. Let $\{e^{-tH}\}_{t>0}$ be the semigroup generated by
$H$. Then for any $f\in L^2(\rn)$, we have
$$e^{-tH}(f):=\sum_{\az\in\nn_0^n}e^{-(2|\az|+n)t}\langle f,\textbf{h}_\az\rangle \textbf{h}_\az$$
with $\langle f,\textbf{h}_\az\rangle:=\int_\rn \textbf{h}_\az(x)f(x)\,dx$.
The kernel $W_t(x,y)$ of semigroup
$e^{-tH}$ is given by, for any $t\in(0,\fz)$ and any $x,\ y\in \rn$,
\begin{equation}\label{5-27a}
W_t(x,y)=\frac1{\pi^{n/2}}\left( \frac{e^{-2t}}{1-e^{-4t}}\r)^{n/2}
e^{-\frac14\big(\frac{1+e^{-2t}}{1-e^{-2t}}|x-y|^2+\frac{1-e^{-2t}}{1+e^{-2t}}|x+y|^2\big)}.
\end{equation}

Considering the square root of $H$, we denote the kernel associated to the operator
$(t\sqrt H)^ke^{-t\sqrt H}$ by $p_{t,k}(x,y)$, $\forall\, t\in(0,\fz)$ and
$\forall\, k\in\nn_0$.
When $k=0$, we drop the subscript $k$ to write $p_t(x,y)$. Then the kernels have the following
estimates (see \cite{MSTZ12} and also \cite[Lemmea 2.1]{BD15}).

\begin{lem}\label{le7-21}
For any given $k\in\nn_0$, there exist positive constants $C$ and $\dz$ such that, for any
$t\in(0,\fz)$ and any $x,\ y\in \rn$,
\begin{enumerate}
\item[{\rm(i)}] $p_t(x,\,y)\le C\frac{t}{(t+|x-y|)^{n+1}};$

\item[{\rm(ii)}] when $k\ge1$, $|p_{t,\,k}(x,\,y)|\le C\frac{t^k}{(t+|x-y|)^{n+k}}$.

\end{enumerate}
\end{lem}

The following important facts was obtained in \cite[Propositions 2.2 and 2.5]{BD15},
which play crucial roles in investigating Besov and Triebel-Lizorkin spaces
associated with the Hermite operator in \cite{BD15,bd21} and also in the present article.

\begin{prop}\label{prop10-16}
Let $k\in\nn_0$ and $t\in(0,\fz)$.
\begin{enumerate}
\item[{\rm (i)}] Then it follows that, for any given $y\in\rn$,
$p_{t,\,k}(\cdot,\,y)\in\cs(\rn)$.

\item[{\rm (ii)}] The operator $(t\sqrt H)^ke^{-t\sqrt H}$ preserves
Schwartz functions in $\cs(\rn)$, namely, for any
$h\in\cs(\rn)$, $(t\sqrt H)^ke^{-t\sqrt H}h\in\cs(\rn)$. More precisely,
for any multi-index $\az,\ \beta\in\nn_0^n$ and any $x\in\rn$,
$$\lf|x^\beta D^\gamma[(t\sqrt H)^ke^{-t\sqrt H}h](x)\r|\ls \min\{1/t,t^{k-1/2}+t^{k+1}\},$$
where the implicit constant is independent of $k$, $t$ and $x$.
\end{enumerate}
\end{prop}

\begin{rem}
Let $f\in\cs'(\rn)$, $k\in\nn_0$ and $t\in(0,\fz)$.
\begin{enumerate}
\item[(i)]
By Proposition \ref{prop10-16}(i),
it is reasonable to define, for any $x\in\rn$,
$$(t\sqrt H)^ke^{-t\sqrt H}f(x):=\langle f,p_{t,k}(x,\cdot)\rangle,$$
where $\langle \cdot,\cdot\rangle$ is the pair between a liner functional in $\cs'(\rn)$
and a function in $\cs(\rn)$. Particularly, we have, for any $x\in\rn$,
$$e^{-\sqrt H}f(x)=\langle f,p_1(x,\cdot)\rangle\ls \lf\langle f, \frac1{(1+|x-\cdot|)^{n+1}}
\r\rangle.$$

\item[(ii)] It follows from Proposition \ref{prop10-16}(ii) that
we can also define $(t\sqrt H)^ke^{-t\sqrt H}f$ as a distribution in $\cs'(\rn)$ by setting
$$\lf\langle (t\sqrt H)^ke^{-t\sqrt H}f,\phi\r\rangle
:=\lf\langle f, (t\sqrt H)^ke^{-t\sqrt H}\phi\r\rangle, \quad \forall\phi\in\cs(\rn).$$
Moreover, this coincides with the definition in (i) of this remark.
\end{enumerate}
\end{rem}

\subsection{Variable Function Spaces}

For a measurable function $p(\cdot):\rn\rightarrow(0,\fz]$, let
\begin{equation*}
p^-:=\esinf_{x\in\rn}p(x)\ \ {\rm and} \ \ p^+:=\esup_{x\in\rn}p(x).
\end{equation*}
Denote by $\cp(\rn)$ the set of measurable functions $p(\cdot)$ with
$0< p^-\le p^+<\fz$. For any $p(\cdot)\in \cp(\rn)$,
the \emph{variable Lebesgue space} $L^{p(\cdot)}(\rn)$ is defined to be the
set of all measurable functions $f$ such that
$$\|f\|_{L^{p(\cdot)}(\rn)}:=\inf\lf\{\lz\in(0,\fz):\ \int_\rn\lf[\frac{|f(x)|}{\lz}
\r]^{p(x)}\,dx\le 1\r\}<\fz.$$

A measurable function $g$ on $\rn$ is said to satisfy the \emph{locally
log-H\"older continuous condition}, denoted by $g\in C_{\loc}^{\log}(\rn)$,
if there exists a positive constant $C_{\log}(g)$ such that, for all $x, y\in\rn$,
\begin{equation*}
|g(x)-g(y)|\le\frac{C_{\log}(g)}{{\log}(e+1/|x-y|)};
\end{equation*}
and $g$ is to said to satisfy the \emph{globally log-H\"older continuous condition},
denoted by $g\in C^{\log}(\rn)$, if $g\in C_{\loc}^{\log}(\rn)$ and there exist
a $g_{\fz}\in\rr$ and a positive constant $C_{\fz}(g)$ such that, for all $x\in\rn$,
$$|g(x)-g_{\fz}|\le\frac{C_{\fz}(g)}{{\log}(e+|x|)}.$$
For any $r\in(0,\fz)$, let $L_\loc^r(\rn)$ be the
\emph{set} of all locally $r$-order integrable functions on $\rn$.

\begin{rem}\label{rem5-20}
\begin{enumerate}
\item[(i)] If $p(\cdot)\in C^{\log}(\rn)$ with $p^-\in(1,\fz)$, then the Hardy-Littlewood
maximal operator
$\cm$ is bounded on the space $L^{p(\cdot)}(\rn)$ (see \cite[Theorem 3.16]{cfbook}), where
the Hardy-Littlewood maximal function $\mathcal{M}$ is defined by setting,
for any $f\in L_{\rm loc}^1(\rn)$ and $x\in\rn$,
$$\mathcal{M}f(x):=\sup_{x\in B}\frac1{|B|}\int_B|f(y)|dy,$$
with the supremum taking over all balls $B\subset\rn$ containing $x$.

\item[(ii)] We remark that the Fefferman-Stein vector-valued inequality on $L^p(\ell^q(\rn))$
corresponding to the Hardy-Littlewood maximal operator is a very important tool
in the study of classical Triebel-Lizorkin space. This is also true on
$L^{p(\cdot)}(\ell^q(\rn))$ with $q$ being a constant, namely, if $q\in(1,\fz)$ and
$p(\cdot)\in C^{\log}(\rn)$ with $p^-\in(1,\fz)$, then there exists a positive constant
$C$ such that, for all sequences $\{f_v\}_{v=1}^\fz$ in $L_{\loc}^1(\rn)$,
$$\lf\|\lf(\sum_{v=1}^{\fz}({\mathcal{M}f_v})^q\r)^{1/q}\r\|_{L^{p(\cdot)}(\rn)}
\le C\lf\|\lf(\sum_{v=1}^{\fz}|f_v|^q\r)^{1/q}\r\|_{L^{p(\cdot)}(\rn)}.$$
Unfortunately, it turns
out that it is not possible if $q$ is not a constant, even if $p$ is a constant
(see \cite[p.\,1746]{DHR09}).
As a substitute, the vector-valued convolutional inequality on
the mixed Lebesgue sequence space $L^{p(\cdot)}(\ell^{q(\cdot)}(\rn))$
developed in \cite[Theorem 3.2]{DHR09} (see also Lemma \ref{lem10-17a} below)
supply well remedy for this absence.
\end{enumerate}
\end{rem}

For any $R\in(0,\fz)$, $v\in\nn_0$ and any $x\in\rn$, we define
$$\eta_{v,R}(x)=\frac{2^{nv}}{(1+2^v|x|)^{R}}.$$
Then the vector-valued convolutional inequality on
the mixed Lebesgue sequence spaces developed in \cite[Theorem 3.2]{DHR09} is stated as follows.

\begin{lem}\label{lem10-17a}
Let $R\in(n,\fz)$ and $p(\cdot),q(\cdot)\in C^{\log}(\rn)$ with $p^-,\ q^-\in(1,\fz)$.
Then there exists a positive constant $C$ such that,
for every sequence $\{f_v\}_{v\in\nn_0}$ of $L^1_{\loc}(\rn)$ functions,
$$\lf\|\lf[\sum_{v\in\nn_0}\lf|\eta_{v,R}\ast f_v\r|^{q(\cdot)}\r]^{1/q(\cdot)}
\r\|_{L^{p(\cdot)}(\rn)}
\leq C\lf\|\lf[\sum_{v\in\nn_0}|f_v|^{q(\cdot)}\r]^{1/q(\cdot)}\r\|_{L^{p(\cdot)}(\rn)}.$$
\end{lem}

Now, let us introduce the inhomogeneous variable Triebel-Lizorkin spaces
associated to the Hermite operator.

\begin{defn}\label{de9-16a}
Let $H$ be the Hermite operator as in \eqref{3-22a}, $p(\cdot),q(\cdot)\in C^{\log}(\rn)$,
$\az(\cdot)\in C_{\loc}^{\log}(\rn)\cap L^\fz(\rn)$
and $m\in\nn$.
Then the \emph{inhomogeneous variable Triebel-Lizorkin spaces}
associated with the Hermite operator $H$, denoted by
$F_{p(\cdot),q(\cdot)}^{\az(\cdot),H,m}(\rn)$,
is defined to be set of all $f\in \cs'(\rn)$ satisfying
$\|f\|_{F^{\az(\cdot),H,m}_{p(\cdot),q(\cdot)}(\rn)}<\fz$, where
\begin{align*}
\|f\|_{F^{\az(\cdot),H,m}_{p(\cdot),q(\cdot)}(\rn)}:=\lf\|e^{-\sqrt{H}}f\r\|_{L^{p(\cdot)}(\rn)}+
\lf\|\lf\{\int_0^1\lf[t^{-\az(\cdot)}|(t\sqrt{H})^m
e^{-t\sqrt{H}}f|\r]^{q(\cdot)}\frac{dt}t\r\}^{1/q(\cdot)}\r\|_{L^{p(\cdot)}(\rn)}.
\end{align*}
\end{defn}

\begin{rem}
\begin{enumerate}
\item[(i)]
The space $F_{p(\cdot),q(\cdot)}^{\az(\cdot),H,m}(\rn)$ is independent of
the choice on the parameter
\begin{equation}\label{7-5a}
m\in \lf(\max\{\az^+,0\}+n+\lf\lfloor \frac n{\min\{1,p^-,q^-\}} -n\r\rfloor+1+C_{\log}(\alpha),
\fz\r),
\end{equation}
 see Theorem \ref{thm10-18} below.
Hence, we always choose $m$ as in \eqref{7-5a} and denote
$F_{p(\cdot),q(\cdot)}^{\az(\cdot),H,m}(\rn)$ simply by
$F_{p(\cdot),q(\cdot)}^{\az(\cdot),H}(\rn)$.

\item[(ii)] When $p(\cdot)\equiv p$, $q(\cdot)\equiv q$ and $\alpha(\cdot)\equiv \az$ are
constants, then the space introduced in Definition \ref{de9-16a} goes back to
the inhomogeneous Triebel-Lizorkin space $F_{p,q}^{\az,H,m}(\rn)$ defined
in \cite[p.\,436]{BD15} but in which $p,\ q$ are required to be in $[1,\fz)$.

\item[(iii)]  We point out that, in \cite[Definition 4.1]{BD15}, Bui and Duong also define
the Triebel-Lizorkin space $F_{p,q}^{\az,H,m}(\beta,\gamma)$ associated with the Hermite operator,
where $\az+n<m$, $p\in(0,\fz)$, $q\in(0,\fz]$ and $0<\beta<1<\gamma$ with $\gamma-\beta>1$,
to be the set of all $f\in\cs'(\rn)$ satisfying
$$\|f\|_{F_{p,q}^{\az,H,m}(\beta,\gamma)}
:=\int_{\beta}^\gamma\|e^{-t\sqrt{H}}f\|_{L^p(\rn)}\,\frac{dt}t
+\lf\|\lf\{\int_0^1[t^{-\az}|(t\sqrt H)^me^{-t\sqrt H}f|]^q\,\frac{dt}t\r\}^{1/q}
\r\|_{L^p(\rn)}<\fz.$$
Using the molecular decomposition, it was obtained in \cite[Theorem 4.4]{BD15}
that the spaces $F_{p,q}^{\az,H,m_1}(\beta_1,\gamma_1)$ and
$F_{p,q}^{\az,H,m_2}(\beta_2,\gamma_2)$ coincide with equivalent norms
for $p,\ q> \frac n{n+1}$.
Moreover, this space coincides with the space $F_{p,q}^{\az,H,m}(\rn)$ mentioned in (ii) of
this remark (see Proposition 4.5 in \cite{BD15}).
\end{enumerate}
\end{rem}

\section{Molecular Characterizations\label{se-3}}

In this section, via some estimates of kernels associated to the Hermite operator,
we establish the molecular characterization of the Triebel-Lizorkin
space $F_{p(\cdot),q(\cdot)}^{\az(\cdot),H,m}(\rn)$. As a consequence,
we show that the space $F_{p(\cdot),q(\cdot)}^{\az(\cdot),H,m}(\rn)$ is independent of
the parameter $m$.

We first recall the definition of the molecule
associated to the Hermite operator $H$ and begin with some notations on dyadic cubes.
For any $v\in\zz$ and $k\in\zz^n$,
denote by $Q_{vk}$ the \emph{dyadic cube} $2^{-v}([0,1)^n+k)$,
$x_{Q_{vk}}:=2^{-v}k$
its \emph{lower left corner} and $\ell(Q_{vk})$ its \emph{side length}.
Let
$$\cd:=\{Q_{vk}:\ v\in\zz,\ k\in\zz^n\}$$
and,
for any $v\in\zz$, we set
$$\cd_v=\{Q\in\cd:\,\ell(Q)=2^{-v}\}.$$

\begin{defn}\label{def12-1}
Let $M,N\in\nn$.
A function $a$ is said to be an \emph{$(H,M,N)$-molecule} associated to the Hermite
operator $H$ if there
exist a function $b$ from the domain of the operator $(\sqrt{H})^M$
and a dyadic cube $Q\in\cd$ such that
\begin{enumerate}
\item[(i)] $a=(\sqrt{H})^Mb$;

\item[(ii)] $|(\sqrt{H})^kb(x)|\le\ell(Q)^{M-k}|Q|^{-\frac12}
\lf(1+\frac{|x-x_Q|}{\ell(Q)}\r)^{-n-N}$, for all $k =0,\cdots,2M$.
\end{enumerate}
If a function $a$ satisfies (i) and (ii) only for
$k=M,\cdots,2M$, then we say that $a$ is an \emph{$(H,M,N)$-zero level molecule}.

When $N=1$, we briefly write an \emph{$(H,M)$-molecule} (resp. \emph{$(H,M)$-zero level molecule})
instead of an \emph{$(H,M,1)$-molecule} (resp. \emph{$(H,M,1)$-zero level molecule}).
\end{defn}

\begin{rem}
We point out that the definition of molecule in Definition \ref{def12-1}
has a little bit difference from that in \cite[Definition 3.3]{BD15},
in which the authors consider the constant case and the size condition (ii) is replaced by
\begin{equation*}
|(\sqrt{H})^kb(x)|\le\ell(Q)^{M-k}|Q|^{\az/n-1/p}
\lf(1+\frac{|x-x_Q|}{\ell(Q)}\r)^{-n-N},
\end{equation*}
where $\az\in\rr$ and $p\in(0,\fz)$ both are constants.
The modification in Definition \ref{def12-1} makes it more convenient to deal with the
variable exponents $\az(\cdot)$ and $p(\cdot)$ in our setting.
\end{rem}

The corresponding sequence space to the variable Triebel-Lizorkin space is given as follows,
which is introduced in \cite{DHR09}.

\begin{defn}
Let $\az(\cdot),p(\cdot),q(\cdot)$ be as in Definition {\rm{\ref{de9-16a}}}.
Then the \emph{sequence space} $f_{p(\cdot),q(\cdot)}^{\az(\cdot)}(\rn)$
is defined to be the set of
all sequence $\{s_Q\}_{Q\in\cd_v,v\in\nn_0}\subset \cc$ such that
$$\lf\|\{s_Q\}_{Q\in\cd_v,v\in\nn_0}\r\|_{f_{p(\cdot),q(\cdot)}^{\az(\cdot)}(\rn)}
:=\lf\|\lf[\sum_{v\in\nn_0}\lf(2^{v\az(\cdot)}\sum_{Q\in\cd_v}|s_Q||Q|^{-\frac12}
\textbf{1}_Q\r)^{q(\cdot)}\r]^{1/q(\cdot)}\r\|_{L^{p(\cdot)}(\rn)}<\fz.$$
\end{defn}

Now we state the molecular characterization of the space
$F_{p(\cdot),q(\cdot)}^{\alpha(\cdot),H,m}(\rn)$.

\begin{thm}\label{thm10-18a}
Let $H$, $\az(\cdot),p(\cdot),q(\cdot)$ be as in Definition {\rm{\ref{de9-16a}}}
and $M,\ N,\ m\in\nn$.
If $f\in F_{p(\cdot),q(\cdot)}^{\az(\cdot),H,m}(\rn)$, then $f$ has the following
decomposition
\begin{equation}\label{04-12}
f=\sum_{v\in\nn_0}\sum_{Q\in\cd_v}s_Qm_Q,
\end{equation}
which converges in $\cs'(\rn)$, where
\begin{enumerate}
\item[$\bullet$] $\{m_Q\}_{Q\in\cd_v,v\in\nn }$
is a sequence of $(H,M,N)$-molecules,

\item[$\bullet$] $\{m_Q\}_{Q\in\cd_0}$
is a sequence of $(H,M)$-zero level molecules,

\item[$\bullet$] $\{s_Q\}_{Q\in\cd_v,v\in\nn_0}$ is a sequence in
$f_{p(\cdot),q(\cdot)}^{\az(\cdot)}(\rn)$.
\end{enumerate}Moreover, it holds true
\begin{equation}\label{eq10-9}
\lf\|\{s_Q\}_{Q\in\cd_v,v\in\nn_0}\r\|_
{f_{p(\cdot),q(\cdot)}^{\az(\cdot)}(\rn)}\le C
\|f\|_{F^{\az(\cdot),H,m}_{p(\cdot),q(\cdot)}(\rn)}
\end{equation}
with $C$ being a constant independent of $f$.
\end{thm}

To prove Theorem \ref{thm10-18a},
we need the following Calder\'on reproducing formula,
which comes from \cite[p.\,437]{BD15} and is a
variant of \cite[Proposition 2.9]{BD15}.
For the completeness, we also present the proof of this formula here via \cite[Proposition 2.8]{BD15}
(see also Lemma \ref{prop10-16a} below).

\begin{prop}\label{prop10-9}
Let $H$ be the Hermite operator as in \eqref{3-22a},
$l\in\nn$ and $f\in\cs'(\rn)$. Then it holds true that
\begin{equation}\label{3-22b}
f=\frac{2^l}{(l-1)!}\int_0^1(t\sqrt{H})^{l}e^{-2t\sqrt{H}}f\frac{dt}t
+\sum_{k=0}^{l-1}\frac{2^k}{k!}
(\sqrt{H})^ke^{-2\sqrt{H}}f
\end{equation}
in $\cs'(\rn)$.
\end{prop}

The following properties are just \cite[Proposition 2.8]{BD15}.

\begin{lem}\label{prop10-16a}
Let $H$ be the Hermite operator as in \eqref{3-22a} and $f\in\cs'(\rn)$. Then
\begin{enumerate}
\item[{\rm(i)}]
for any $k\in\nn_0$, $\lim_{t\rightarrow\fz}(t\sqrt{H})^ke^{-t\sqrt{H}}f=0$ in $\cs'(\rn);$

\item[{\rm(ii)}] for any $k\in\nn$,
$\lim_{t\rightarrow0}(t\sqrt{H})^ke^{-t\sqrt{H}}f=0$ in $\cs'(\rn);$

\item[{\rm(iii)}] $\lim_{t\rightarrow0}(I-e^{-t\sqrt{H}})f=0$ in $\cs'(\rn).$
\end{enumerate}
\end{lem}

\begin{proof}[Proof of Proposition \ref{prop10-9}]
Using integration by part and some properties of semigroup, we find that,
for any $\dz\in(0,\fz)$,
\begin{align}\label{eq10-16}
\int_\delta^1(t\sqrt{H})^{l}e^{-2t\sqrt{H}}f\frac{dt}t
&=-\sum_{k=0}^{l-1}\frac{(l-1)!}{2^{l-k}k!}(t\sqrt{H})^ke^{-2t\sqrt{H}}f
\Big{|}_{\dz}^{1}\\
&=-\sum_{k=0}^{l-1}\frac{(l-1)!}{2^{l-k}k!}
\lf[(\sqrt{H})^ke^{-2\sqrt{H}}f-\dz^k(\sqrt{H})^ke^{-2\dz\sqrt{H}}f\r].\noz
\end{align}
By Lemma \ref{prop10-16a}, we know that, for any $k\in\nn$,
$$\lim_{\dz\rightarrow0}(\dz\sqrt{H})^ke^{-2\dz\sqrt{H}}f=0\quad
{\rm and}\quad \lim_{\dz\rightarrow0}e^{-2\dz\sqrt{H}}f=f$$
both in $\cs'(\rn)$.
Letting $\dz\rightarrow0$ in \eqref{eq10-16}, we obtain
$$\int_0^1(t\sqrt{H})^{l}e^{-2t\sqrt{H}}f\frac{dt}t
=-\sum_{k=0}^{l-1}\frac{(l-1)!}{2^{l-k}k!}(\sqrt{H})^ke^{-2\sqrt{H}}f
+\frac{(l-1)!}{2^l}f,\quad {\rm in}\quad \cs'(\rn),$$
which implies \eqref{3-22b}. This finishes the proof of Proposition \ref{prop10-9}.
\end{proof}

\begin{proof}[Proof of Theorem \ref{thm10-18a}]
Let $f\in F^{\az(\cdot),H,m}_{p(\cdot),q(\cdot)}(\rn)$.
We will complete the proof by three steps.

\emph{Step 1)} \emph{We establish the decomposition \eqref{04-12}}.
By Proposition \ref{prop10-9} with $l=m+M+N$, we find that
$$f=\frac{2^{m+M+N}}{(m+M+N-1)!}\int_0^1(t\sqrt{H})^{M+N}e^{-t\sqrt{H}}
(t\sqrt{H})^me^{-t\sqrt{H}}f\frac{dt}t+\sum_{k=0}^{m+M+N-1}\frac{2^k}{k!}
(\sqrt{H})^ke^{-2\sqrt{H}}f,$$
which converges in $\cs'(\rn)$.
For each $v\in\nn$ and $Q\in\cd_v$, let
$$s_Q:=|Q|^{\frac12}\sup_{(y,t)\in Q\times[2^{-v-1},2^{-v})}|(t\sqrt{H})^me^{-t\sqrt{H}}f(y)|
\quad{\rm and}\quad m_Q:=(\sqrt{H})^Mb_Q,$$
 where
$$b_Q:=\frac{2^{m+M+N}}{(m+M+N-1)!}\frac1{s_Q}\int_{2^{-v-1}}^{2^{-v}}
t^Me^{-t\sqrt{H}}(t\sqrt{H})^N\lf[(t\sqrt{H})^me^{-t\sqrt{H}}f\cdot\textbf{1}_Q\r]\frac{dt}t.$$
For each $Q\in\cd_0$, we set
$$s_Q:=\sup_{y\in Q}|e^{-\sqrt{H}}f(y)|\quad{\rm and}\quad
m_Q:=\frac1{s_Q}\sum_{k=0}^{m+M+N-1}\frac{2^k}{k!}(\sqrt{H})^k
e^{-\sqrt{H}}\lf(e^{-\sqrt{H}}f\cdot\textbf{1}_Q\r).$$
Then it is easy to see that \eqref{04-12} holds true.

\emph{Step 2)} \emph{We prove that $\{m_Q\}_{Q\in\cd_v,v\in\nn_0}$ are molecules}.
Considering $\{m_Q\}_{Q\in \cd_v}$ with $v\in\nn$,
by Lemma \ref{le7-21}, we know that, for any $k=0,\cdots,2M$ and $x\in \rn$,
\begin{align*}
\lf|(\sqrt{H})^kb_Q(x)\r|
&\sim\lf|\frac1{s_Q}
\int_{2^{-v-1}}^{2^{-v}}t^{M-k}(t\sqrt{H})^{N+k}e^{-t\sqrt{H}}
\lf[(t\sqrt{H})^me^{-t\sqrt{H}}f\cdot\textbf{1}_Q\r](x)\frac{dt}t\r|\\
&\ls\frac1{s_Q}\int_{2^{-v-1}}^{2^{-v}}t^{M-k}
\int_Q|p_{t,N+k}(y,x)||(t\sqrt{H})^me^{-t\sqrt{H}}f(y)|dy\frac{dt}t\\
&\ls\frac1{s_Q}
\int_{2^{-v-1}}^{2^{-v}}\int_Qt^{M-k}|(t\sqrt{H})^me^{-t\sqrt{H}}f(y)|
\frac{t}{(t+|x-y|)^{n+N}}dy\frac{dt}t\\
&\ls|Q|^{-\frac12+\frac{M-k}n}\lf[1+\frac{|x-x_Q|}{\ell(Q)}\r]^{-n-N}.
\end{align*}
This implies that, for any $Q\in \cd_v$ with $v\in\nn$,
$m_Q$ is a multiple of an $(H,M,N)$-molecule.

For $m_Q$ with $Q\in\cd_0$, we first observe that
$$m_Q=(\sqrt H)^M\lf(\frac1{s_Q}\sum_{k=0}^{m+M+N-1}\frac{2^k}{k!}(\sqrt H)^{k-M}
e^{-\sqrt{H}}\lf(e^{-\sqrt H}f\cdot\textbf{1}_Q\r)\r)=:(\sqrt H)^M\widetilde{b}_Q.$$
Moreover, by Lemma \ref{le7-21} again, we find that, for any $i=M,\cdots,2M$
and $x\in \rn$,
\begin{align*}
|(\sqrt H)^i \widetilde{b}(x)|&=\lf|\frac1{s_Q}\sum_{k=0}^{m+M+N-1}\frac{2^k}{k!}
(\sqrt{H})^{k+i-M}e^{-\sqrt{H}}(e^{-\sqrt{H}}f\cdot\textbf{1}_Q)(x)\r|\\
&\ls\frac1{s_Q}\sum_{k=0}^{m+M+N-1}\frac{2^k}{k!}\int_Q\frac{1}{(1+|x-y|)^{n+\max\{k+i-M,1\}}}
|e^{-\sqrt{H}}f(y)|dy\\
&\ls\int_Q\frac1{(1+|x-y|)^{n+1}}dy\\
&\ls\lf(1+|x-x_Q|\r)^{-n-1}.
\end{align*}
Therefore, in this case, $m_Q$ is a multiple of an $(H,M)$-zero level molecule.

\emph{Step 2)} \emph{We now prove \eqref{eq10-9}}.
To this end, it suffices to prove
\begin{align}\label{eq10-9a}
\|\{s_Q\}_{Q\in\cd_v,v\in\nn}\|_{f_{p(\cdot),q(\cdot)}^{\alpha(\cdot)}(\rn)}
\ls\lf\|\lf[\int_0^1\lf(t^{-\az(\cdot)}|(t\sqrt{H})^me^{-t\sqrt{H}}f|
\r)^{q(\cdot)}\frac{dt}t\r]^{1/q(\cdot)}\r\|_{L^{p(\cdot)}(\rn)}
\end{align}
and
\begin{equation}\label{eq10-9b}
\lf\|\sum_{Q\in\cd_0}|s_Q|\textbf{1}_Q\r\|_{L^{p(\cdot)}(\rn)}
\ls\lf\|e^{-\sqrt{H}}f\r\|_{L^{p(\cdot)}(\rn)}.
\end{equation}

We first consider \eqref{eq10-9a}. Indeed, since
$(\sqrt{H})^me^{-t\sqrt{H}}f$ is a solution of the equation
$(-\Delta_{x,t}+|x|^2)u=0$, where $\Delta_{x,t}(u)=u_{tt}+\Delta u$,
it follows that $|(\sqrt{H})^me^{-t\sqrt{H}}f|^2$ is subharmonic (see \cite[p.\,1988]{AB07}).
Due to \cite[Lemma 5.2]{AB07}, it is clear to see that
\begin{equation}\label{eq10-17}
\sup_{(y,t)\in\wz Q}|(\sqrt{H})^me^{-t\sqrt{H}}f(y)|\ls
\lf(\frac1{|\wz Q|}\int_{\frac32\wz Q}|(\sqrt{H})^me^{-t\sqrt{H}}f(y)|^rdydt\r)^{1/r},
\end{equation}
where $\wz Q=Q\times[2^{-v-1},2^{-v})$ is a cube in $\rr^{n+1}$ and $0<r<\min\{1,p^-,q^-\}$.
Moreover,
$$\sup_{(y,t)\in\wz Q}|(t\sqrt{H})^me^{-t\sqrt{H}}f(y)|
\ls\lf(\frac1{|Q|}\int_{\frac342^{-v-1}}^{\frac982^{-v}}\int_{\frac32Q}
|(t\sqrt{H})^me^{-t\sqrt{H}}f(y)|^rdy\frac{dt}t\r)^{1/r}.$$
Therefore, by the Minkowski inequality and choosing $L>n+rC_{\log}(\az)$ large enough,
we obtain, for any $x\in Q$,
\begin{align*}
|Q|^{-\frac12}|s_Q|
&\ls\lf(\frac1{|Q|}\int_{\frac342^{-v-1}}^{\frac982^{-v}}\int_{\frac32Q}
|(t\sqrt{H})^me^{-t\sqrt{H}}f(y)|^rdy\frac{dt}t\r)^{1/r}\\ \noz
&\sim\lf\{\int_{\frac32Q}\frac{2^{vn}}{(1+2^v|x-y|)^{L}}
\lf[\int_{\frac342^{-v-1}}^{\frac982^{-v}}
|(t\sqrt{H})^me^{-t\sqrt{H}}f(y)|^r\frac{dt}t\r]dy\r\}^{1/r}\\ \noz
&\ls\lf\{\eta_{v,L}\ast\lf[\int_{\frac342^{-v-1}}^{\frac982^{-v}}
|(t\sqrt{H})^me^{-t\sqrt{H}}f|^r\frac{dt}t\r]\r\}^{1/r}(x).\noz
\end{align*}
This, together with \cite[Lemma 6.1]{DHR09}, further implies that, for any $x\in Q$,
\begin{align*}
&2^{v\az(x)}\sum_{Q\in\cd_v}|Q|^{-\frac12}|s_Q|\textbf{1}_Q(x)\\
&\hs\ls\sum_{Q\in\cd_v}\lf\{\eta_{v,L-rC_{\log}(\alpha)}\ast
\lf[\int_{\frac342^{-v-1}}^{\frac982^{-v}}
2^{rv\az(\cdot)}|(t\sqrt{H})^me^{-t\sqrt{H}}f|^r\frac{dt}t\r]\r\}^{1/r}(x)\cdot\textbf{1}_Q(x)\\
&\hs\ls\lf\{\eta_{v,L-rC_{\log}(\alpha)}\ast
\lf[\int_{\frac342^{-v-1}}^{\frac982^{-v}}\big(2^{v\az(\cdot)}
|(t\sqrt{H})^me^{-t\sqrt{H}}f|\big)^r\frac{dt}t\r]\r\}^{1/r}(x).
\end{align*}
Using Lemma \ref{lem10-17a} and the facts that $p^-/r>1$, $q^-/r>1$, we then find that
\begin{align*}
&\|\{s_Q\}_{Q\in\cd_v,v\in\nn}\|_{f_{p(\cdot),q(\cdot)}^{\alpha(\cdot)}(\rn)}\\
&\quad\ls\lf\|\lf\{\sum_{v\in\nn }\lf\{\eta_{v,L-rC_{\log}(\alpha)}\ast
\lf[\int_{\frac342^{-v-1}}^{\frac982^{-v}}
\big(t^{-\az(\cdot)}|(t\sqrt{H})^me^{-t\sqrt{H}}f|\big)^r\frac{dt}t\r]\r\}^{q(\cdot)/r}
\r\}^{1/q(\cdot)}\r\|_{L^{p(\cdot)}(\rn)}\\
&\quad\ls\lf\|\lf\{\sum_{v\in\nn }\lf[\int_{\frac342^{-v-1}}^{\frac982^{-v}}
\big(t^{-\az(\cdot)}|(t\sqrt{H})^me^{-t\sqrt{H}}f|\big)^r\frac{dt}t\r]^{q(\cdot)/r}
\r\}^{1/q(\cdot)}\r\|_{L^{p(\cdot)}(\rn)}.
\end{align*}
By the H\"older inequality, we finally conclude that
\begin{align*}
\|\{s_Q\}_{Q\in\cd_v,v\in\nn}\|_{f_{p(\cdot),q(\cdot)}^{\alpha(\cdot)}(\rn)}
&\ls\lf\|\lf[\sum_{v\in\nn }\int_{\frac342^{-v-1}}^{\frac982^{-v}}\lf(t^{-\az(\cdot)}
|(t\sqrt{H})^me^{-t\sqrt{H}}f|\r)^{q(\cdot)}\frac{dt}t\r]^{1/q(\cdot)}\r\|_{L^{p(\cdot)}(\rn)}\\
&\ls\lf\|\lf[\int_0^1\lf(t^{-\az(\cdot)}
|(t\sqrt{H})^me^{-t\sqrt{H}}f|\r)^{q(\cdot)}\frac{dt}t\r]^{1/q(\cdot)}\r\|_{L^{p(\cdot)}(\rn)}.
\end{align*}

To prove \eqref{eq10-9b}, we conclude by an argument similar to that used in
\eqref{eq10-17} that $|e^{-\sqrt{H}}f|^2$ is subharmonic. Thus, for any
$Q\in\cd_0$ and any $x\in Q$,
\begin{align*}
|s_Q|&=\sup_{y\in Q}|e^{-\sqrt{H}}f(y)|
\ls\lf(\frac1{|\frac32Q|}\int_{\frac32Q}|e^{-\sqrt{H}}f(y)|^{\wz{r}}dy\r)^{1/\wz{r}}
\ls\lf[\mathcal{M}\lf(|e^{-\sqrt{H}}f|^{\wz{r}}\r)(x)\r]^{1/{\wz{r}}},
\end{align*}
where $0<\wz{r}<\min\{1,p^-\}$. From this and the boundedness of Hardy-Littlewood maximal operator
on variable Lebesgue space(see Remark \ref{rem5-20}(i)), we deduce that
\begin{align*}
\lf\|\sum_{Q\in\cd_0}|s_Q|\textbf{1}_Q\r\|_{L^{p(\cdot)}(\rn)}&\ls
\lf\|\lf[\mathcal{M}\lf(|e^{-\sqrt{H}}f|^{\wz r}\r)(x)\r]^{1/{\wz r}}\r\|_{L^{p(\cdot)}(\rn)}\ls
\lf\|e^{-\sqrt{H}}f\r\|_{L^{p(\cdot)}(\rn)}.
\end{align*}
Therefore, \eqref{eq10-9a} and \eqref{eq10-9b} hold true,
which completes the proof of Theorem \ref{thm10-18a}.
\end{proof}

\begin{thm}\label{thm10-9}
Let $\az(\cdot),p(\cdot),q(\cdot)$ be as in Definition {\rm \ref{de9-16a}}
and $m$ as in \eqref{7-5a}. Suppose that
$N\in \nn\cap(\frac n{\min\{1,p^-,q^-\}}-n+C_{\log}(\alpha),\fz)$
satisfies $m>\max\{\az^+,0\}+N+n$, and $ M\in\nn$ with $M>n+N+\max\{\az^-,-\az^-\}$.
If
$$f:=\sum_{v\in\nn_0}\sum_{Q\in\cd_v}s_Qm_Q$$
converges in $\cs'(\rn)$, where
\begin{enumerate}
\item[$\bullet$] $\{m_Q\}_{Q\in\cd_v,v\in\nn }$ is a sequence of $(H,M,N)$-molecules,
\item[$\bullet$] $\{m_Q\}_{Q\in\cd_0}$ is a sequence of $(H,M)$-zero level molecules,
\item[$\bullet$] $\{s_Q\}_{Q\in\cd_v,v\in\nn_0}$ is a sequence in
$f_{p(\cdot),q(\cdot)}^{\az(\cdot)}(\rn)$,
\end{enumerate}
then $f$ belongs to the space $F_{p(\cdot),q(\cdot)}^{\az(\cdot),H,m}(\rn)$.
Moreover, we have
$$\|f\|_{F_{p(\cdot),q(\cdot)}^{\az(\cdot),H,m}(\rn)}\le C
\lf\|\{s_Q\}_{Q\in\cd_v,v\in\nn_0}\r\|_{f_{p(\cdot),q(\cdot)}^{\az(\cdot)}(\rn)}$$
with $C$ being a constant independent of the sequence $\{s_Q\}_{Q\in\cd_v,v\in\nn_0}$.
\end{thm}

To present the proof of Theorem \ref{thm10-9}, we need some lemmas.
The following one is just from \cite[Lemma 5.13]{CK21}.

\begin{lem}\label{le7-21b}
Let $0<\tau\leq1$, $j, v\in\nn_0$ and $R\in(n/\tau,\fz)$. Then it follows that,
for any sequence $\{s_Q\}_{Q\in\cd_v}\subset\cc$ and $x\in\rn$,
$$
\sum_{Q\in\cd_v}|s_Q|\lf(1+2^{\min(v,j)}|x-x_Q|\r)^{-R}
\ls\max\lf\{1,2^{(v-j)R}\r\}\lf[\eta_{v,R\tau}\ast
\lf(\sum_{Q\in\cd_v}|s_Q|^\tau\textbf{1}_Q\r)\r]^{1/\tau}(x),
$$
where the implied constant is independent of $v,\ j,\ x$ and $\{s_Q\}_{Q\in\cd_v}$.
\end{lem}

The following estimates were proved in \cite[Lemma 3.6]{BD15}, which play key roles in
the proof of Theorem \ref{thm10-9}.

\begin{lem}\label{le7-21c}
Under the assumptions as in Theorem \ref{thm10-9}, there exists a positive constant $C$ such that,
for any $Q\in\cd_v$ with $v\in\nn_0$ and $x\in\rn$, we have
$$|(t\sqrt{H})^me^{-t\sqrt{H}}m_Q(x)|\le C|Q|^{-\frac12}
\lf(\frac{t}{2^{-v}}\r)^{m-N-n}\lf(1+\frac{|x-x_Q|}{2^{-v}}\r)^{-n-N},\, \forall t\in(0,2^{-v}],$$
and
$$|(t\sqrt{H})^me^{-t\sqrt{H}}m_Q(x)|\le C|Q|^{-\frac12}
\lf(\frac{2^{-v}}{t}\r)^M\lf(1+\frac{|x-x_Q|}{t}\r)^{-n-N},\, \forall t\in(2^{-v},\fz).$$
\end{lem}

\begin{proof}[Proof of Theorem \ref{thm10-9}]
We only need to prove
\begin{align}\label{eq7-19}
{\rm I}:=\lf\|\lf[\int_0^1\lf(t^{-\az(\cdot)}
|(t\sqrt{H})^me^{-t\sqrt{H}}f|\r)^{q(\cdot)}\frac{dt}t
\r]^{1/q(\cdot)}\r\|_{L^{p(\cdot)}(\rn)}
\ls\|\{s_Q\}_{Q\in \cd_v,v\in\nn_0}\|_{f_{p(\cdot),q(\cdot)}^{\alpha(\cdot)}(\rn)}
\end{align}
and
\begin{equation}\label{eq7-19c}
\|e^{-\sqrt{H}}f\|_{L^{p(\cdot)}(\rn)}\ls
\|\{s_Q\}_{Q\in \cd_v,v\in\nn_0}\|_{f_{p(\cdot),q(\cdot)}^{\alpha(\cdot)}(\rn)}.
\end{equation}

We first prove \eqref{eq7-19}. We write
$${\rm I}
=\lf\|\lf[\sum_{k\in\nn }\int_{2^{-k}}^{2^{-k+1}}\lf(t^{-\az(\cdot)}
|(t\sqrt{H})^me^{-t\sqrt{H}}f|\r)^{q(\cdot)}\frac{dt}t\r]^{1/q(\cdot)}
\r\|_{L^{p(\cdot)}(\rn)},$$
and then divide it into two parts
\begin{align}\label{eq7-19a}
{\rm I}
&\ls\lf\|\lf[\sum_{k\in\nn }\lf(2^{k\az(\cdot)}\sum_{v<k}\sum_{Q\in\cd_v}
|s_Q|\sup_{t\in[2^{-k},2^{-k+1})}|(t\sqrt{H})^me^{-t\sqrt{H}}m_Q|
\r)^{q(\cdot)}\r]^{1/q(\cdot)}\r\|_{L^{p(\cdot)}(\rn)}\\ \noz
&\quad+\lf\|\lf[\sum_{k\in\nn }\lf(2^{k\az(\cdot)}\sum_{v\ge k}\sum_{Q\in\cd_v}
|s_Q|\sup_{t\in[2^{-k},2^{-k+1})}|(t\sqrt{H})^me^{-t\sqrt{H}}m_Q|
\r)^{q(\cdot)}\r]^{1/q(\cdot)}\r\|_{L^{p(\cdot)}(\rn)}\\ \noz
&=:\rm I_1+I_2.\noz
\end{align}
For $\rm I_1$, since $v<k$ in this case and
$t<2^{-k+1}\le 2^{-v}$, it follows from Lemma \ref{le7-21c} that, for any $x\in Q$,
\begin{align}\label{eq9-18c}
&\sum_{Q\in\cd_v}|s_Q|\sup_{t\in[2^{-k},2^{-k+1})}|(t\sqrt{H})^me^{-t\sqrt{H}}m_Q(x)|\\
&\hs\hs\ls 2^{-(k-v)(m-N-n)}\sum_{Q\in\cd_v}|Q|^{-\frac12}|s_Q|
\lf(1+\frac{|x-x_Q|}{2^{-v}}\r)^{-n-N}.\noz
\end{align}
Moreover,
we apply Lemma \ref{le7-21b} with
$n/(n+N-C_{\log}(\alpha))<\tau<\min\{1,p^-,q^-\}$ and $R=n+N$ to obtain
\begin{equation}\label{eq9-18b}
\sum_{Q\in\cd_v}|Q|^{-\frac12}|s_Q|\lf(1+\frac{|x-x_Q|}{2^{-v}}\r)^{-n-N}
\ls\lf[\eta_{v,(n+N)\tau}\ast\lf(\sum_{Q\in\cd_v}|Q|^{-\frac12\tau}
|s_Q|^{\tau}\textbf{1}_Q(\cdot)\r)\r]^{1/{\tau}}(x).
\end{equation}

Now, inserting \eqref{eq9-18b} into \eqref{eq9-18c}, together with \cite[Lemma 6.1]{DHR09},
yields
\begin{align*}
&\sum_{Q\in\cd_v}|s_Q|\sup_{t\in[2^{-k},2^{-k+1})}|(t\sqrt{H})^me^{-t\sqrt{H}}m_Q(x)|\\
&\hs\hs\ls2^{-v\az(x)}2^{(k-v)(m-N-n)}\lf[\eta_{v,(n+N)\tau-\tau C_{\log}(\alpha)}\ast
\Bigg(2^{\tau v\az(\cdot)}\sum_{Q\in\cd_v}|Q|^{-\frac12\tau}
|s_Q|^{\tau}\textbf{1}_Q(\cdot)\Bigg)\r]^{1/{\tau}}(x).
\end{align*}
Then
\begin{align*}
{\rm I}_1&\ls\lf\|\lf[\sum_{k\in\nn }\lf\{\sum_{v<k}2^{-(k-v)(m-N-n-\az^+)\tau}\r.\r.\r.\\
&\qquad\qquad\times
\lf.\lf.\lf.\lf[\eta_{v,(n+N)\tau-\tau C_{\log}(\alpha)}\ast\Bigg(2^{\tau v\az(\cdot)}\sum_{Q\in\cd_v}
|Q|^{-\frac12\tau}
|s_Q|^{\tau}\textbf{1}_Q(\cdot)\Bigg)\r]\r\}^{{q(\cdot)}/\tau}\r]^{1/{q(\cdot)}}
\r\|_{L^{p(\cdot)}(\rn)}.
\end{align*}

Since $m>\max\{\az^+,0\}+n+N$, $\sum_{k>v}2^{-(k-v)(m-N-n-\az^+)q(\cdot)}$
is bounded by a constant independent of $k,v$. Thus, by Lemma \ref{lem10-17a}, we have
\begin{align}\label{eq5-16a}
{\rm I}_1
&\ls\lf\|\lf\{\sum_{v\in\nn }\lf[\eta_{v,(n+N)\tau-\tau C_{\log}(\alpha)}\ast
\Bigg(2^{\tau v\az(\cdot)}\sum_{Q\in\cd_v}
|Q|^{-\frac12\tau}|s_Q|^{\tau}\textbf{1}_Q(\cdot)\Bigg)\r]^{{q(\cdot)}/{\tau}}
\r\}^{1/{q(\cdot)}}\r\|_{L^{p(\cdot)}(\rn)}\\ \noz
&\ls\lf\|\lf\{\sum_{v\in\nn }\Bigg(2^{\tau v\az(\cdot)}\sum_{Q\in\cd_v}
|Q|^{-\frac12\tau}|s_Q|^{\tau}\textbf{1}_Q(\cdot)\Bigg)^{{q(\cdot)}/{\tau}}
\r\}^{1/q(\cdot)}\r\|_{L^{p(\cdot)}(\rn)}\\ \noz
&\ls\lf\|\{s_Q\}_{Q\in\cd_v,v\in\nn_0}\r\|_{f_{p(\cdot),q(\cdot)}^{\az(\cdot)}(\rn)}.\noz
\end{align}

It remains to deal with $\rm I_2$, via an argument similar to that used in $\rm I_1$. We only
 give some different point involved
Lemma \ref{le7-21c} for $v\ge k$, i.e.
\begin{equation*}
\sup_{t\in[2^{-k},2^{-k+1})}|(t\sqrt{H})^me^{-t\sqrt{H}}m_Q(x)|
\ls|Q|^{-\frac12}2^{-(v-k)M}\lf(1+\frac{|x-x_Q|}{2^{-k}}\r)^{-n-N}.
\end{equation*}
Noticing that $M>\az^-+n+N$, we conclude that
\begin{align*}
{\rm I_2}&\ls\lf\|\lf[\sum_{k\in\nn }\lf\{\sum_{v\ge k}2^{-(v-k)(M-n-N+\az^-)}\r.\r.\r.\\
&\qquad\qquad\times
\lf.\lf.\lf.\lf[\eta_{v,(n+N)\tau-\tau C_{\log}(\alpha)}\ast\lf(2^{\tau v\az(\cdot)}\sum_{Q\in\cd_v}
|Q|^{-\frac12\tau}|s_Q|^{\tau}
\textbf{1}_Q(\cdot)\r)\r]^{1/{\tau}}\r\}^{q(\cdot)}\r]^{1/q(\cdot)}\r\|_{L^{p(\cdot)}(\rn)}\\
&\ls\lf\|\lf[\sum_{v\in\nn }
\lf(2^{v\az(\cdot)}\sum_{Q\in\cd_v}|Q|^{-\frac12}
|s_Q|\textbf{1}_Q\r)^{q(\cdot)}\r]^{1/q(\cdot)}\r\|_{L^{p(\cdot)}(\rn)}\\
&\ls\lf\|\{s_Q\}_{Q\in\cd_v,v\in\nn_0}\r\|_{f_{p(\cdot),q(\cdot)}^{\az(\cdot)}(\rn)}.
\end{align*}
This, combined with \eqref{eq7-19a} and \eqref{eq5-16a}, completes the proof of \eqref{eq7-19}.

We now give the proof of \eqref{eq7-19c}.
By Lemma \ref{le7-21}, we know that,
for any $Q\in\cd_v,v\in\nn$ and $x\in\rn$,
\begin{align}\label{5-25a}
|e^{-\sqrt{H}}m_Q(x)|&=|(\sqrt{H})^Me^{-\sqrt{H}}b_Q(x)|
\le\int_\rn|p_{1,M}(y, x)||b_Q(y)|dy\\
&\ls\int_\rn\frac 1{(1+|x-y|)^{n+M}}\ell(Q)^M|Q|^{-\frac12}\lf(1+\frac{|y-x_Q|}{\ell(Q)}
\r)^{-n-N}dy\noz\\
&\ls \int_\rn \frac{(1+\frac{|x-y|}{\ell(Q)})^{n+N}}{(\frac1{\ell(Q)}
+\frac{|x-y|}{\ell(Q)})^{n+M-\delta}}\,dy\times
\frac{\ell(Q)^{\delta-n}|Q|^{-1/2}}{(1+\frac{|x-x_Q|}{\ell(Q)})^{n+N}}\noz\\
&\ls \frac{2^{v(n-\delta)}|Q|^{-1/2}}{(1+2^v|x-x_Q|)^{n+N}},\noz
\end{align}
where $\delta$ is chosen such that $\delta+\alpha^->0$ and
$M>n+N+\delta$, and, for any $Q\in\cd_0$,
\begin{align*}
|e^{-\sqrt{H}}m_Q(x)|
\le\int_\rn|p_1(y, x)||m_Q(y)|dy
\ls(1+|x-x_Q|)^{-n-N}.
\end{align*}
Thus, by this, \eqref{5-25a}, Lemma \ref{le7-21b} and \cite[Lemma 6.1]{DHR09},
we obtain
\begin{align*}
|e^{-\sqrt{H}}f(x)|
\ls\sum_{v\in\nn_0}2^{-v[\delta+\alpha(x)]}
\lf\{\eta_{v,(n+N)r-C_{\log}(\az)r}\ast\lf(\Bigg[\sum_{Q\in\cd_v}2^{v\az(\cdot)}|Q|^{-\frac12}
|s_Q|\textbf{1}_Q\Bigg]^r\r)(x)\r\}^{\frac1r},
\end{align*}
where $r\in(0,\min\{1,p^-,q^-\})$.
For any $x\in\rn$, applying the H\"older inequality if $q(x)\in(1,\fz)$
and the fact that $(\sum_v|a_v|)^{q(x)}\le \sum_v|a_v|^{q(x)}$ if $q(x)\in(0,1]$,
we find that
\begin{align*}
|e^{-\sqrt{H}}f(x)|
\ls \lf[\sum_{v\in\nn_0}
\lf\{\eta_{v,(n+N)r-C_{\log}(\az)r}\ast\lf(\Bigg[\sum_{Q\in\cd_v}2^{v\az(\cdot)}|Q|^{-\frac12}
|s_Q|\textbf{1}_Q\Bigg]^r\r)(x)\r\}^{\frac{q(x)}r}\r]^{\frac1{q(x)}}.
\end{align*}
Therefore, involving Lemma \ref{lem10-17a}, we conclude that \eqref{eq7-19c} is true.
This completes the proof of Theorem \ref{thm10-9}.
\end{proof}

As a consequence of Theorems \ref{thm10-18a} and \ref{thm10-9}, we show that the space
$F_{p(\cdot),q(\cdot)}^{\az(\cdot),H,m}(\rn)$ is independent of the choice
of $m$ when $m$ is large enough. Precisely, we have the following result.
\begin{thm}\label{thm10-18}
Let $\az(\cdot)$, $p(\cdot)$, $q(\cdot)$ and $H$ be as in Definition {\rm{\ref{de9-16a}}}.
If $m_i\in\nn $ is large as in \eqref{7-5a},
then the space $F_{p(\cdot),q(\cdot)}^{\az(\cdot),H,m_1}(\rn)$ and
$F_{p(\cdot),q(\cdot)}^{\az(\cdot),H,m_2}(\rn)$ coincide with equivalent (quasi-)norms.
\end{thm}

\begin{proof}
Let $f\in F_{p(\cdot),q(\cdot)}^{\az(\cdot),H,m_1}(\rn)$. Then, by Theorem \ref{thm10-18a},
we fine that there exist a sequence of $(H,M,N)$-molecules
$\{m_Q\}_{Q\in\cd_v,v\in\nn }$, a sequence of $(H,M)$-zero level molecules
$\{m_Q\}_{Q\in\cd_0}$, and a sequence of coefficients $\{s_Q\}_{Q\in\cd_v,v\in\nn_0}$,
where $M\in\nn$ large enough, such that
$$f=\sum_{v\in\nn_0}\sum_{Q\in\cd_v}s_Qm_Q\ \ {\rm in}\ \ \cs'(\rn),$$
and
\begin{equation*}
\lf\|\{s_Q\}_{Q\in\cd_v,v\in\nn_0}\r\|_{f_{p(\cdot),q}^{\az(\cdot)}(\rn)}\ls
\|f\|_{F_{p(\cdot),q(\cdot)}^{\az(\cdot),H,m_1}(\rn)}.
\end{equation*}
Moreover, by Theorem \ref{thm10-9}, we know that
$f\in F_{p(\cdot),q(\cdot)}^{\az(\cdot),H,m_2}(\rn)$ and
\begin{equation*}
\|f\|_{F^{\az(\cdot),H,m_2}_{p(\cdot),q(\cdot)}(\rn)}\ls
\lf\|\{s_Q\}_{Q\in\cd_v,v\in\nn_0}\r\|_{f_{p(\cdot),q(\cdot)}^{\az(\cdot)}(\rn)}
\ls\|f\|_{F^{\az(\cdot),H,m_1}_{p(\cdot),q(\cdot)}(\rn)}.
\end{equation*}
This implies
$F_{p(\cdot),q(\cdot)}^{\az(\cdot),H,m_1}(\rn)\hookrightarrow
F_{p(\cdot),q(\cdot)}^{\az(\cdot),H,m_2}(\rn)$.
Similarly, we have
$F_{p(\cdot),q(\cdot)}^{\az(\cdot),H,m_2}(\rn)\hookrightarrow
F_{p(\cdot),q(\cdot)}^{\az(\cdot),H,m_1}(\rn)$,
which completes the proof of Theorem \ref{thm10-18}.
\end{proof}

We end this section by the following proposition.

\begin{prop}\label{prop6-21}
Let $H$, $\az(\cdot),p(\cdot),q(\cdot)$ be as in Definition {\rm{\ref{de9-16a}}} with
$p^-,\ q^-\in(1,\fz)$. Then
$$\cs(\rn)\hookrightarrow F_{p(\cdot),q(\cdot)}^{\az(\cdot),H}(\rn)
\hookrightarrow \cs'(\rn).$$

\begin{proof}
Let $f\in \cs(\rn)$. Then it follows from Proposition \ref{prop10-16} that,
for any $k\in\nn_0$ and $t\in (0,\fz)$,
$(t\sqrt H)^ke^{-t\sqrt H}f\in \cs(\rn)$ and, for any $g\in\nn_0^n$ and $x\in\rn$,
\begin{equation}\label{6-7a}
|D^\gamma (t\sqrt H)^ke^{-t\sqrt H}f(x)|\ls \min\{1/t,t^{k-1/2}+t^{k+1}\}(1+|x|)^{-L}
\end{equation}
with any $L\in(0,\fz)$. Thus, it is easy to see that
$\|e^{-\sqrt H} f\|_{L^{p(\cdot)}(\rn)}<\fz$. On the other hand,
using \eqref{6-7a}, we find that, for any $x\in \rn$,
\begin{align*}
&\lf\{\int_0^1\lf[t^{-\alpha(\cdot)}|(t\sqrt H)^ke^{-t\sqrt H}f(x)|\r]^{q(\cdot)}\,
\frac{dt}t\r\}^{\frac1{q(\cdot)}}\\
&\hs \ls \lf\{\int_0^1\lf[t^{-\alpha(\cdot)}(t^{m-\frac12}+t^{m+1})\r]^{q(\cdot)}\,
\frac{dt}t\r\}^{\frac1{q(\cdot)}}(1+|x|)^{-L}
\ls (1+|x|)^{-L}.
\end{align*}
By this and choosing $L$ large enough, we conclude that
$$\lf\|\lf\{\int_0^1\lf[t^{-\az(\cdot)}|(t\sqrt{H})^m
e^{-t\sqrt{H}}f|\r]^{q(\cdot)}\frac{dt}t\r\}^{1/q(\cdot)}\r\|_{L^{p(\cdot)}(\rn)}<\fz,$$
which implies that $f\in F_{p(\cdot),q(\cdot)}^{\az(\cdot),H}(\rn)$.

To prove the second ``$\hookrightarrow$", we only need to show that, for any $\phi\in \cs(\rn)$,
$|\langle f,\phi\rangle|<\fz$. By the Calder\'on reproducing formula in Proposition
\ref{prop10-9}, we have
\begin{align*}
|\langle f,\phi\rangle|
&\ls\lf|\lf\langle\int_0^1(t\sqrt{H})^{l}e^{-2t\sqrt{H}}f\frac{dt}t,\phi\r\rangle\r|
+\lf|\lf\langle\sum_{k=0}^{l-1}\frac{2^k}{k!}
(\sqrt{H})^ke^{-2\sqrt{H}}f,\phi\r\rangle\r|\\
&=:{\rm I}_1+{\rm I}_2,
\end{align*}
where $l:=l_1+l_2$ is large enough.
For ${\rm I}_1$, by the H\"older inequality with
$$\frac1{p(\cdot)}+\frac1{p'(\cdot)}=1,\quad \frac1{q(\cdot)}+\frac1{q'(\cdot)}=1,$$
and Proposition \ref{prop10-16}, we know that
\begin{align*}
{\rm I}_1
&\ls \int_{\rn}\int_0^1|(t\sqrt H)^{l_1}e^{-t\sqrt H}f(x)|
|(t\sqrt H)^{l_2}e^{-t\sqrt H}\phi(x)|\,\frac{dt}t\,dx\\
&\ls \int_{\rn}\lf\{\int_0^1|t^{-\alpha(x)}
(t\sqrt H)^{l_1}e^{-t\sqrt H}f(x)|^{q(x)}
\,\frac{dt}t\r\}^{\frac1{q(x)}} \lf\{\int_0^1|t^{\alpha(x)}(t\sqrt H)^{l_2}e^{-t\sqrt H}\phi(x)|^{q'(x)}
\,\frac{dt}t\r\}^{\frac1{q'(x)}}\,dx\\
&\ls \lf\|\lf\{\int_0^1|t^{-\alpha(\cdot)}(t\sqrt H)^{l_1}e^{-t\sqrt H}f|^{q(\cdot)}
\,\frac{dt}t\r\}^{\frac1{q(\cdot)}}\r\|_{L^{p(\cdot)}(\rn)}
\lf\|(1+|\cdot|)^{-R}\r\|_{L^{p'(\cdot)}(\rn)}\\
&\ls \|f\|_{F_{p(\cdot),q(\cdot)}^{\az(\cdot),H}(\rn)},
\end{align*}
where $R$ is chosen large enough.
Similarly, we also have ${\rm I}_2\ls \|f\|_{F_{p(\cdot),q(\cdot)}^{\az(\cdot),H}(\rn)}$,
which completes the proof of Proposition \ref{prop6-21}.
\end{proof}
\end{prop}

\section{Boundedness of Singular Integrals on
$F_{p(\cdot),q(\cdot)}^{\az(\cdot),H}(\rn)$\label{se-4}}

As applications, we develop the boundedness of singular integrals
associated to $H$ and the spectral multipliers of Laplace type
for the Hermite operator in this section.

\subsection{Regularity of Fractional Hermite Equation}

In this part, we consider the regularity of solutions of two fractional Hermite
equations:
$$H^{\sigma}u=f\ \ {\rm and}\ \ (I+H)^{\sigma}u=f,\ \ {\rm on}\ \ \rn,$$
for any $\sigma\in(0,\fz)$ and $f\in F_{p(\cdot),q(\cdot)}^{\az(\cdot),H}(\rn)$. To solve the
indicated equations, it is enough to investigate operators
$H^{-\sigma}$ and $(I+H)^{-\sigma}$,
named by the Riesz potential and the Bessel potential of the Hermite
operator, respectively.

The following lemma comes from \cite[Proposition 2.5]{BD15}.
\begin{lem}\label{lem10-19a}
Assume that $\phi\in\cs(\rn)$. Then, for all $\sigma\in(0,\fz)$,
we have $H^{-\sigma}\phi\in\cs(\rn)$.
\end{lem}

Moreover, by adapting the argument used in \cite[Proposition 2.5]{BD15},
we find that, for any $\sigma\in(0,\fz)$ and $\phi\in\cs(\rn)$,
\begin{equation}\label{5-26a}
H^{-\sigma}\phi=\frac1{\Gamma(\sigma)}\int_0^{\fz}t^{\sigma}e^{-tH}\phi\frac{dt}t.
\end{equation}
For any $\sigma\in(0,\fz)$, thanks to Lemma \ref{lem10-19a},we define the Hermite-Riesz potential
$H^{-\sigma}:\cs'(\rn)\rightarrow\cs'(\rn)$ by duality
$$\langle H^{-\sigma}f,\phi\rangle:=\langle f,H^{-\sigma}\phi\rangle$$
for all $f\in\cs'(\rn)$ and $\phi\in\cs(\rn)$.

Recall that $W_t(\cdot,\cdot)$ is given by \eqref{5-27a} for all
$t\in(0,\fz)$. Denote by $W_{t,k}(\cdot,\cdot)$
the kernel of $(tH)^ke^{-tH}$ for all $k\in\nn_0$.

\begin{rem}
For any $t\in(0,\fz)$ and $x,y\in\rn$,
following \cite[Lemma 2.4]{BD15}, we see that
\begin{equation}\label{eq10-20}
W_t(x,y)\le\frac{C}{t^{n/2}}\exp\lf(-c\frac{|x-y|^2}t\r),
\end{equation}
moreover, by \cite[Lemma 2.5]{CD00}, we also have,
\begin{equation}\label{eq10-20a}
W_{t,k}(x,y)\le\frac{c}{t^{n/2}}\exp\lf(-c\frac{|x-y|^2}t\r),
\end{equation}
where $C,\ c$ are constants independent of $t,\ x$ and $y$.
\end{rem}

\begin{thm}\label{thm10-25}
Let $\az(\cdot),p(\cdot),q(\cdot)$ and $H$ be as in Definition {\rm{\ref{de9-16a}}}.
For any $\sigma\in(n/2,\ \fz)$, the operator $H^{-\sigma}$ is bounded from
$F_{p(\cdot),q(\cdot)}^{\az(\cdot),H}(\rn)$ to
$F_{p(\cdot),q(\cdot)}^{\az(\cdot)+2\sigma,H}(\rn)$.
\end{thm}

\begin{proof}
Let $f\in F_{p(\cdot),q(\cdot)}^{\az(\cdot),H}(\rn)$,
$M,\ N\in\nn$ large enough.
By Theorem \ref{thm10-18a}, $f$ has the decomposition
$$f=\sum_{v\in\nn_0,Q\in \cd_v}s_Qm_Q,\quad {\rm in}\quad \cs'(\rn),$$
where  $\{m_Q\}_{v\in\nn,Q\in \cd_v}$ are $(H,4M,N)$-molecules,
$\{m_Q\}_{Q\in \cd_0}$ are $(H,4M)$-zero level molecules and the sequence
$\{s_Q\}_{v\in\nn_0,Q\in \cd_v}\in f_{p(\cdot),q(\cdot)}^{\az(\cdot)}(\rn)$ satisfying
$$\|\{s_Q\}_{v\in\nn_0,Q\in \cd_v}\|_{f_{p(\cdot),q(\cdot)}^{\az(\cdot)}(\rn)}
\ls \|f\|_{F_{p(\cdot),q(\cdot)}^{\az(\cdot),H}(\rn)}.$$
By the definition of $H^{-\sigma}$, we find that
$$H^{-\sigma}f=\sum_{v\in\nn_0,Q\in \cd_v}s_QH^{-\sigma}m_Q
=\sum_{v\in\nn_0,Q\in \cd_v}[|Q|^{2\sigma/n}s_Q]\lf[|Q|^{-2\sigma/n}H^{-\sigma}m_Q\r].$$
Therefore, to prove Theorem \ref{thm10-25}, it is enough to show that
$|Q|^{-2\sigma/n}H^{-\sigma}m_Q$ is a multiple of $(H,2M,N)$-molecule.

Since $m_Q$, with $Q\in\cd_v,\ v\in\nn$, is a $(H,4M,N)$-molecule, it follows that there
exists a function $b_Q$ such that $m_Q=(\sqrt{H})^{4M}b_Q$.
Then
$$H^{-\sigma}m_Q=(\sqrt{H})^{2M}\lf[H^{-\sigma}(\sqrt{H})^{2M}b_Q\r].$$
Next we claim that
\begin{equation}\label{eq10-19}
|(\sqrt{H})^k|Q|^{-2\sigma/n}H^{-\sigma}(\sqrt{H})^{2M}b_Q(x)|\ls \ell(Q)^{2M-k}|Q|^{-\frac12}
\lf(1+2^v|x-x_Q|\r)^{-n-N},
\end{equation}
for all $k=0,\cdots,4M$. Indeed, from \eqref{5-26a}, we deduce that, for any $x\in\rn$,
\begin{align}\label{7-4a}
\lf|H^{-\sigma}\lf(|Q|^{-2\sigma/n}(\sqrt{H})^{2M+k}b_Q\r)(x)\r|
&=\lf|\frac1{\Gamma(\sigma)}\int_0^{\fz}t^{\sigma}
e^{-tH}\lf(|Q|^{-2\sigma/n}(\sqrt{H})^{2M+k}b_Q\r)(x)\frac{dt}t\r|\\
&\ls\int_0^{4^{-v}}\lf|t^{\sigma}e^{-tH}\lf(|Q|^{-2\sigma/n}(\sqrt{H})^{2M+k}
b_Q\r)(x)\r|\frac{dt}t\noz\\
&\hs\hs+\int_{4^{-v}}^{\fz}\lf|t^{\sigma}e^{-tH}\lf(|Q|^{-2\sigma/n}(\sqrt{H})^{2M+k}b_Q\r)(x)\r|
\frac{dt}t\noz\\
&=:{\rm J}_1(x)+{\rm J}_2(x)\noz.
\end{align}
For ${\rm J}_1(x)$, by \eqref{eq10-20}, we have
\begin{align*}
&\lf|e^{-tH}\lf(|Q|^{-2\sigma/n}(\sqrt{H})^{2M+k}b_Q\r)(x)\r|\\
&\quad=\int_\rn|W_t(y,x)||Q|^{-2\sigma/n}|(\sqrt{H})^{2M+k}b_Q(y)|dy\\
&\quad\ls\int_\rn\frac{1}{t^{n/2}}\exp\lf(-c\frac{|x-y|^2}t\r)
|Q|^{-2\sigma/n}|(\sqrt{H})^{2M+k}b_Q(y)|dy\\
&\quad\ls\int_\rn\frac{1}{t^{n/2}}\lf(1+\frac{|x-y|}{\sqrt{t}}\r)^{-n-N}
|Q|^{-2\sigma/n}2^{-v(2M-k)}|Q|^{-\frac12}\lf(1+\frac{|y-x_Q|}{2^{-v}}\r)^{-n-N}dy.
\end{align*}
Noting that $0<t<4^{-v}$ in this case, it follows from Lemma \ref{le7-21c} that
$$\lf|e^{-tH}\lf(|Q|^{-2\sigma/n}(\sqrt{H})^{2M+k}b_Q\r)(x)\r|\ls|Q|^{-2\sigma/n}2^{-v(2M-k)}|Q|^{-\frac12}
\lf(1+\frac{|x-x_Q|}{2^{-v}}\r)^{-n-N}\lf(\frac{2^{-v}}{\sqrt{t}}\r)^n.$$
Thus, as long as $\sigma>n/2$, we have
\begin{align}\label{eq10-20b}
{\rm J}_1(x)&\ls|Q|^{-2\sigma/n}2^{-v(2M-k+n)}|Q|^{-\frac12}\lf(1+\frac{|x-x_Q|}{2^{-v}}\r)^{-n-N}
\int_0^{4^{-v}}t^{\sigma-n/2}\frac{dt}t\\
&\ls2^{-v(2M-k)}|Q|^{-\frac12}\lf(1+\frac{|x-x_Q|}{2^{-v}}\r)^{-n-N}\noz.
\end{align}

For ${\rm J}_2(x)$, by \eqref{eq10-20a}, we know that
\begin{align*}
&\lf|e^{-tH}\lf(|Q|^{-2\sigma/n}(\sqrt{H})^{2M+k}b_Q\r)(x)\r|\\
&\quad=t^{-M}\int_\rn|W_{t,M}(y,x)|Q|^{-2\sigma/n}(\sqrt{H})^kb_Q(y)|dy\\
&\quad\ls t^{-M}\int_\rn\frac{1}{t^{n/2}}\exp\lf(-c\frac{|x-y|^2}t\r)
|Q|^{-2\sigma/n}|(\sqrt{H})^kb_Q(y)|dy\\
&\quad\ls t^{-M}\int_\rn\frac{1}{t^{n/2}}\lf(1+\frac{|x-y|}{\sqrt{t}}\r)^{-n-N}
|Q|^{-2\sigma/n}2^{-v(4M-k)}|Q|^{-\frac12}\lf(1+\frac{|y-x_Q|}{2^{-v}}\r)^{-n-N}dy.
\end{align*}
By this, Lemma \ref{le7-21c} and the fact $t\ge4^{-v}$ in this case, we conclude that
\begin{align*}
&\lf|e^{-tH}\lf(|Q|^{-2\sigma/n}(\sqrt{H})^{2M+k}b_Q\r)(x)\r|\\
&\quad\ls t^{-M}2^{-v(4M-k)}
|Q|^{-\frac12}\lf(1+\frac{|x-x_Q|}{\sqrt{t}}\r)^{-n-N}\\
&\quad\ls (4^v t )^{(n+N)/2}t^{-M}|Q|^{-2\sigma/n}2^{-v(4M-k)}|Q|^{-\frac12}
\lf(1+\frac{|x-x_Q|}{2^{-v}}\r)^{-n-N},
\end{align*}
which implies that
\begin{align*}
{\rm J}_2(x)&\ls|Q|^{-2\sigma/n}2^{-v(4M-k-n-N)}
|Q|^{-\frac12}\lf(1+\frac{|x-x_Q|}{2^{-v}}\r)^{-n-N}
\int_{4^{-v}}^{\fz}t^{\sigma+(n+N)/2-M}\frac{dt}t\\
&\ls2^{-v(2M-k)}|Q|^{-\frac12}\lf(1+\frac{|x-x_Q|}{2^{-v}}\r)^{-n-N}.
\end{align*}
This, combined with \eqref{7-4a} and \eqref{eq10-20b}, implies \eqref{eq10-19} and then
$|Q|^{-2\sigma/n}H^{-\sigma} m_Q$ is a multiple of $(H,2M,N)$ molecule.
Similarly, for $Q\in\cd_0$, if $m_Q$ is a $(H,4M)$-zero level molecule,
$H^{-\sigma}m_Q$ is a multiple $(H,2M)$-zero level molecule.
This completes the proof of Theorem \ref{thm10-25}.
\end{proof}

\begin{rem}
By an argument as in Theorem \ref{thm10-25},
we obtain the boundedness of the Bessel potential
of the Hermite operator $(I+H)^{-\sigma}$ on variable Triebel-Lizorkin spaces
$F_{p(\cdot),q(\cdot)}^{\az(\cdot),H}(\rn)$, the details being omitted here.
\end{rem}

\subsection{Boundedness of Spectral Multipliers\label{s-SM}}

In this subsection, we study the spectral multipliers of Laplace type for the Hermite
operators in the following form:
\begin{equation}\label{eq10-25}
m(H):=\int_0^{\fz}\phi(t)He^{-tH}dt,
\end{equation}
where $\phi\in L^{\fz}(\rr)$ is given.

Following the same discussions of \cite[Proposition 2.5]{BD15}, we immediately find
that $m(H)\psi\in\cs(\rn)$ in the wake of $\psi\in\cs(\rn)$. Hence, for
$f\in\cs'(\rn)$, $m(H)f$ can be defined as a functional in $\cs'(\rn)$ by setting
$$\langle m(H)f, \psi\rangle:=\langle f, m(H)\psi\rangle.$$

We state the main result of this subsection as follows.
\begin{thm}\label{thm5-27}
Let $\az(\cdot),p(\cdot),q(\cdot)$ and $H$ be as in Definition {\rm{\ref{de9-16a}}}.
The spectral multiplier $m(H)$ in \eqref{eq10-25} is bounded
on $F_{p(\cdot),q(\cdot)}^{\az(\cdot),H}(\rn)$.
\end{thm}
\begin{proof}
The proof of this theorem is similar to that of Theorem \ref{thm10-25}.
For similarity, we only prove that, for any $(H,4M,N)$-(zero level)
molecule $m_Q=H^{2M}b_Q$, $m(H)m_Q$ is an $(H,2M,N)$-(zero level) molecule
associated to the same dyadic cube $Q\in\cd_v$ for some $v\in\nn$ ($v=0$).

Indeed, when $v\in\nn$, we write $m(H)m_Q=H^M\widetilde{b}_Q$, where
$$\widetilde{b}_Q:=\int_0^\fz\phi(t)He^{-tH}H^Mb_Qdt.$$
For $k=0,\cdots,4M$, we have, for any $x\in\rn$,
\begin{align*}
|(\sqrt{H})^k\widetilde{b}_Q(x)|&\le\int_0^{4^{-v}}|\phi(t)e^{-tH}H^{M+k/2+1}b_Q(x)|dt+
\int_{4^{-v}}^\fz|\phi(t)H^{M+1}e^{-tH}(\sqrt{H})^kb_Q(x)|dt\\
&\le\int_0^{4^{-v}}|e^{-tH}H^{M+k/2+1}b_Q(x)|dt+
\int_{4^{-v}}^\fz|H^{M+1}e^{-tH}(\sqrt{H})^kb_Q(x)|dt.
\end{align*}
At this stage, we obtain
$$|(\sqrt{H})^k\widetilde{b}_Q(x)|\ls2^{-v(2M-k)}
|Q|^{-\frac12}\lf(1+\frac{|x-x_Q|}{2^{-v}}\r)^{-n-N}.$$
And the same thing also happens in case of $v=0$. This completes the proof of Theorem \ref{thm5-27}.
\end{proof}

\section{Relationships between Spaces $F_{p(\cdot),q(\cdot)}^{\az(\cdot),H}(\rn)$
and $F_{p(\cdot),q(\cdot)}^{\az(\cdot)}(\rn)$}

In this section, we give an embedding from variable Triebel-Lizorkin space
$F_{p(\cdot),q(\cdot)}^{\az(\cdot)}(\rn)$
introduced by Diening et al. \cite{DHR09}
into the space $F_{p(\cdot),q(\cdot)}^{\az(\cdot),H}(\rn)$,
via the atomic decomposition of the space
$F_{p(\cdot),q(\cdot)}^{\az(\cdot)}(\rn)$ and molecule decomposition of the space
$F_{p(\cdot),q(\cdot)}^{\az(\cdot),H}(\rn)$,.

We begin with the following definition of atoms.

\begin{defn}
Let $k\in\zz$, $l\in\nn_0$ and $M\ge n$. A function $a$ is called
a $(k,l,M)$-smooth atom associated to a dyadic cube $Q\in \cd_v$ with
$v\in\nn_0$, if it satisfies the following conditions for some $\wz m>M$:
\begin{enumerate}
\item[{\rm(i)}] $\supp a\subset3Q$;

\item[{\rm(ii)}] if $v\in\nn$, then $\int_\rn x^{\gamma}a(x)=0$ for all $|\gamma|\le k$;

\item[{\rm(iii)}] $\big|D^{\gamma}a(x)\big|\le2^{|\gamma|v}|Q|^{1/2}
\eta_{v,\wz m}(x-x_Q)$ for all multi-indices $\gamma\in\nn_0^n$ with $|\gamma|\le l$.
\end{enumerate}
The conditions (i), (ii) and (iii) are called the \emph{support}, \emph{moment} and
\emph{decay conditions}, respectively.
\end{defn}

\begin{rem}
The number $M$ needs to be chosen sufficiently large. Since $M$ is allowed depending
on the parameters, we will usually omit it from our notation of atoms.
\end{rem}
The atomic decomposition of the space $F_{p(\cdot),q(\cdot)}^{\az(\cdot)}(\rn)$ is stated as
follows (\cite[Theorem 3.11]{DHR09}).

\begin{lem}\label{lem12-1}
Let $\alpha(\cdot)$, $p(\cdot)$ and $q(\cdot)$ be as in Definition \ref{de9-16a},
and $f\in F_{p(\cdot),q(\cdot)}^{\az(\cdot)}(\rn)$.
For any $\varepsilon>0$, let $K=n/\min\{1,p^-,q^-\}-n+\varepsilon$ and $L=\az^++1+\varepsilon$.
Then there exist a family of
($K, L$)-atoms $\{a_Q\}_{Q\in \cd_v,v\in\nn_0}$ and a sequence of coefficients
$\{s_Q\}_{Q\in \cd_v,v\in\nn_0}\in f_{p(\cdot),q(\cdot)}^{\az(\cdot)}(\rn)$
such that
$$f= \sum_{Q\in \cd_v,v\in\nn_0}s_Qa_Q\quad {\rm in}\quad \cs'(\rn)$$
and $\|\{s_Q\}_{Q\in \cd_v,v\in\nn_0}\|_{f_{p(\cdot),q(\cdot)}^{\az(\cdot)}(\rn)}
\sim \|f\|_{F_{p(\cdot),q(\cdot)}^{\az(\cdot)}(\rn)}$, where the implicit constants are
independent of $f$.
\end{lem}

\begin{lem}\label{le11-20}
Let $K\in\zz$, $L\in\nn\cap (K,\fz)$ and $m\in \nn\cap(n+K,\fz)$. Assume that $a_Q$ is
an $(K,L)$-atom for some cube
$Q\in\cd_v, v\in\zz$. Then we have
\begin{enumerate}
\item[{\rm(i)}] for all $t\le 2^{-v}$ and $x\in\rn$,
$$|(t\sqrt{H})^me^{-t\sqrt{H}}a_Q(x)|\ls|Q|^{-\frac12}
\lf(1+\frac{|x-x_Q|}{2^{-v}}\r)^{-n-K};$$

\item[{\rm(ii)}] for all $t>2^{-v}$ and $x\in\rn$,
$$|(t\sqrt{H})^me^{-t\sqrt{H}}a_Q(x)|\ls|Q|^{-\frac12}
\lf(\frac{2^{-v}}{t}\r)^L\lf(1+\frac{|x-x_Q|}{t}\r)^{-n-K},$$
\end{enumerate}
where the implicit constants are independent of $t$, $Q$ and $x$.
\end{lem}
\begin{proof}
We refer to \cite[Lemma 6.3]{BD15} for the proof of this lemma.
\end{proof}

From Lemmas \ref{lem12-1} and \ref{le11-20},
by some arguments similar to that used in the proof of Theorem \ref{thm10-9},
we have the following result.

\begin{thm}\label{thm12-2}
Let $H$, $p(\cdot)$, $q(\cdot)$, $\az(\cdot)$ be as in
Definition \ref{de9-16a} and $\az^+<0$.
Then we have
$$F_{p(\cdot),q(\cdot)}^{\az(\cdot)}(\rn)\hookrightarrow
F_{p(\cdot),q(\cdot)}^{\az(\cdot),H}(\rn).$$
\end{thm}
\begin{proof}
Suppose $f\in F_{p(\cdot),q(\cdot)}^{\az(\cdot)}(\rn)$. Let
$K=\lfloor\max\{n/{\min\{1,p^-,q^-\}}-n,C_{\log}(\az)\}+\varepsilon\rfloor$ and
$$L=\max\lf\{\az^++1,n/{\min\{1,p^-,q^-\}}-\az^+\r\}+\varepsilon,$$
where $\varepsilon$ is
an arbitrarily small positive number.
Then, by Lemma \ref{lem12-1}, there exist a family of
($K, L$)-atoms $\{a_Q\}_{Q\in \cd_v,v\in\nn_0}$ and a sequence of coefficients
$\{s_Q\}_{Q\in \cd_v,v\in\nn_0}\in f_{p(\cdot),q(\cdot)}^{\az(\cdot)}(\rn)$ such that
$$f= \sum_{Q\in \cd_v,v\in\nn_0}s_Qa_Q\quad {\rm in}\quad \cs'(\rn)$$
and $\|\{s_Q\}_{Q\in \cd_v,v\in\nn_0}\|_{f_{p(\cdot),q(\cdot)}^{\az(\cdot)}(\rn)}
\ls\|f\|_{F_{p(\cdot),q(\cdot)}^{\az(\cdot)}(\rn)}$.
Now we claim that
\begin{equation}\label{eq12-1}
\|f\|_{F_{p(\cdot),q(\cdot)}^{\az(\cdot),H}(\rn)}\ls
\|\{s_Q\}_{Q\in \cd_v,v\in\nn_0}\|_{f_{p(\cdot),q(\cdot)}^{\az(\cdot)}(\rn)}.
\end{equation}
Indeed, via Lemmas \ref{le11-20} and \ref{le7-21b} with
$n/(n+K-C_{\log}(\alpha))<\tau<\min\{1,p^-,q^-\}$ and $R=n+K$, we have, for each $k\in\nn_0$,
\begin{align*}
&\sum_{v\in\nn_0}\sum_{Q\in\cd_v}
|s_Q|\sup_{t\in[2^{-k},2^{-k+1})}|(t\sqrt{H})^me^{-t\sqrt{H}}a_Q|\\
&\quad\le\lf(\sum_{v\in\nn_0:v<k}\sum_{Q\in\cd_v}+\sum_{v\in\nn_0:v\ge k}\sum_{Q\in\cd_v}\r)
\lf(|s_Q|\sup_{t\in[2^{-k},2^{-k+1})}|(t\sqrt{H})^me^{-t\sqrt{H}}a_Q|\r)\\
&\quad\ls\sum_{v<k}2^{-v\az(x)}\lf[\eta_{v,(n+K)\tau-\tau C_{\log}(\alpha)}
\ast\lf(2^{\tau v\az(\cdot)}\sum_{Q\in\cd_v}|Q|^{-\frac12\tau}
|s_Q|^{\tau}\chi_Q(\cdot)\r)\r]^{1/{\tau}}\\
&\qquad+\sum_{v\ge k}2^{-v\az(x)}2^{-(v-k)L}2^{(v-k)(n+K)}\lf[\eta_{v,(n+K)\tau-\tau C_{\log}(\alpha)}
\ast\lf(2^{\tau v\az(\cdot)}\sum_{Q\in\cd_v}|Q|^{-\frac12\tau}
|s_Q|^{\tau}\chi_Q(\cdot)\r)\r]^{1/{\tau}}.
\end{align*}
Hence,
\begin{align*}
&2^{k\az(x)}\sum_{v\in\nn_0}\sum_{Q\in\cd_v}
|s_Q|\sup_{t\in[2^{-k},2^{-k+1})}|(t\sqrt{H})^me^{-t\sqrt{H}}a_Q|\\
&\quad\ls\sum_{v<k}2^{(k-v)\az^+}\lf[\eta_{v,(n+K)\tau-\tau C_{\log}(\alpha)}
\ast\lf(2^{\tau v\az(\cdot)}\sum_{Q\in\cd_v}|Q|^{-\frac12\tau}
|s_Q|^{\tau}\chi_Q(\cdot)\r)\r]^{1/{\tau}}\\
&\qquad+\sum_{v\ge k}2^{-(v-k)(L-K-n+\az^+)}\lf[\eta_{v,(n+K)\tau-\tau C_{\log}(\alpha)}
\ast\lf(2^{\tau v\az(\cdot)}\sum_{Q\in\cd_v}|Q|^{-\frac12\tau}
|s_Q|^{\tau}\chi_Q(\cdot)\r)\r]^{1/{\tau}}.
\end{align*}
By noticing that $\az^+<0$ and $L-K-n+\az^+>0$, so similar to the proof of Theorem \ref{thm10-9},
we claim that \eqref{eq12-1} holds true. This completes the proof of Theorem \ref{thm12-2} .
\end{proof}

\noindent\textbf{Declarations}\quad
The author has no conflict of interest to declare. Also, data sharing not applicable to
this article as no data sets were generated or analysed during the current study.

\medskip

\noindent Qi Sun

\noindent School of Mathematics and Statistics,
Central South University,
Changsha, Hunan 410075, People's Republic of China

\smallskip

\noindent {\it E-mail}: \texttt{qisun2577@csu.edu.cn}

\medskip

\noindent Ciqiang Zhuo (Corresponding author)

\noindent  Key Laboratory of Computing and Stochastic Mathematics
(Ministry of Education), School of Mathematics and Statistics,
Hunan Normal University,
Changsha, Hunan 410081, People's Republic of China

\smallskip

\noindent {\it E-mail}: \texttt{cqzhuo87@hunnu.edu.cn}

\end{document}